\documentclass[11pt]{article}
\usepackage{amssymb}
\usepackage{amsmath,amsthm}
\usepackage[latin1]{inputenc}
\usepackage{hyperref}
\usepackage{enumerate}
\usepackage{color}
\usepackage{graphicx}
\usepackage{tikz}
\usepackage{tkz-graph}
\usepackage{tkz-berge}
\hypersetup{colorlinks=true, linkcolor=blue, citecolor=blue,
urlcolor=blue}
\usepackage{algorithm}
\usepackage{algorithmic}

\newtheorem{prelem}{{\bf Theorem}}

 \newtheorem{theorem}{Theorem}

\newtheorem{proposition}[theorem]{Proposition}

\theoremstyle{definition}

\newtheorem{remark}[theorem]{Remark}

\title{Quasi-total Roman domination in graphs}

\author{Suitberto Cabrera-Garc\'ia$^{1}$, Abel Cabrera-Mart\'{\i}nez$^{2}$, Ismael G. Yero$^{3,}$\footnote{\em Corresponding author.}\\
  \\
$^{1}$ {\small Universitat Polit\'ecnica de Valencia}\\{\small Departamento de Estad\'istica e Investigaci\'on Operativa Aplicadas y Calidad}\\{\small Camino de Vera s/n, 46022 Valencia, Spain.}\\ {\small
suicabga\@@eio.upv.es}\\
$^{2}$ {\small Universitat Rovira i Virgili}\\
{\small Departament d'Enginyeria Inform\`atica i Matem\`atiques } \\  {\small Av. Pa\"{\i}sos
Catalans 26, 43007 Tarragona, Spain.} \\{\small
  abel.cabrera\@@urv.cat}\\
$^{3}$ {\small Universidad de C\'adiz}\\
{\small Departamento de Matem\'aticas, Escuela Polit\'ecnica Superior de Algeciras}\\{\small Av. Ram\'on Puyol s/n, 11202 Algeciras, Spain.}\\ {\small
ismael.gonzalez\@@uca.es}
}

\date{}

\begin{document}
\maketitle

\begin{abstract}
A quasi-total Roman dominating function on a graph $G=(V, E)$ is a function $f : V \rightarrow \{0,1,2\}$ satisfying the following:
\begin{itemize}
  \item every vertex $u$ for which $f(u) = 0$ is adjacent to at least one vertex $v$ for which $f(v) =2$, and
  \item if $x$ is an isolated vertex in the subgraph induced by the set of vertices labeled with 1 and 2, then $f(x)=1$.
\end{itemize}
The weight of a quasi-total Roman dominating function is the value $\omega(f)=f(V)=\sum_{u\in V} f(u)$. The minimum weight of a quasi-total Roman dominating function on a graph $G$ is called the quasi-total Roman domination number of $G$. We introduce the quasi-total Roman domination number of graphs in this article, and begin the study of its combinatorial and computational properties.
\end{abstract}



\section{Introduction}

Domination in graphs is a classical topic, and nowadays one of the most active areas of research in graph theory. This fact can be seen for instance through more than 1600 articles published in the topic (more than 1000 of them in the last 10 years), according to the MathSciNet database with the queries: ``domination number'' or ``dominating sets''. A high number of concepts, open problem, research lines are nowadays open and being dealt with.  The two books \cite{hhs1,hhs2}, although a little not updated by now, contain a large number of the most important and classical results in the topic before this new century. One of the topics which is intensively studied concern the so-called Roman domination, which is a domination version arising from some historical roots coming from the ancient Roman Empire. This article deals precisely with this style of  domination.

We consider $G$ as a non directed graph with vertex set $V(G)$ having neither loops nor multiple edges. A \emph{dominating set} of $G$ is a subset $D\subseteq V(G)$ of vertices such that every vertex not in $D$ is adjacent to at least one vertex in $D$. The minimum cardinality of a dominating set is called the \emph{domination number} of $G$ and is denoted by $\gamma (G)$.  This concept has many applications to several fields, since it naturally arises in facility location problems, in monitoring communication, or in electrical networks. A first variant of it is as follows. A \emph{total dominating set} of $G$ (whether it has no isolated vertex) is a set $S$ of vertices of $G$ such that every vertex is adjacent to a vertex in $S$. Every graph without isolated vertices has a total dominating set, since $S = V(G)$ is such a set. The \emph{total domination number} of $G$, denoted by $\gamma_t(G)$, is the minimum cardinality of a total dominating set. Many variants (like the total domination for instance) of the basic concepts of domination have appeared in the literature. We refer to \cite{hhs1,hhs2}  for a survey of the area. Moreover, for more information on total domination, we suggest the recent book \cite{HeYe_book} or the latest survey \cite{Henning2009}.

A variation of domination called Roman domination was formally introduced in graph theory, by Cockayne et al. in \cite{CDH04}, based on some works from ReVelle \cite{re1,re2} and Stewart \cite{St99}. Also see ReVelle and Rosing \cite{rer} for an integer programming formulation of the problem. The concept of Roman domination can be formulated in terms of graphs as follows. A \emph{Roman dominating function} (RDF for short) on a graph $G$ is a vertex labeling $f : V(G) \rightarrow \{0, 1, 2\}$ such that every vertex with label $0$ has a neighbor with label $2$.
For a RDF $f$, let $ V_i^f = \{v \in V (G) : f(v) = i\}$ for $i = 0, 1, 2$. Since this three sets determine $f$, we can equivalently write  $f=(V_0^f, V_1^f, V_2^f)$ (or simply $f=(V_0, V_1, V_2)$ if there is no confusion). The \emph{weight} $f(V(G))$ of a RDF $f$ on $G$ is the value $\omega(f)=\sum_{v\in V(G)} f(v)$, which equals $|V_1^f| + 2|V_2^f|$.
The \emph{Roman domination number} $\gamma_R(G)$ of $G$  is the minimum weight of a RDF on $G$. A RDF on a graph $G$ with minimum weight will be referred to as a $\gamma_R(G)$-function.
	
As a variation of Roman domination (and also as another variation of domination), Liu and Chang \cite{lc} have introduced the concept of total Roman domination in graphs although in a more general setting. In a more specific style, the total Roman domination number was first presented and deeply studied in \cite{total-Roman-1}. A \emph{total Roman dominating function} (TRDF) on a graph $G$ with no isolated vertex, is a Roman dominating function  $f = (V_0, V_1, V_2)$ on $G$ such that the set $\{v \in  V(G) : f (v) \not = 0\}$ induces a subgraph without isolated vertices. The  \emph{total Roman domination number} $\gamma_{tR} (G)$  is the minimum weight of a TRDF on $G$. A TRDF  with minimum weight in a graph $G$ is a $\gamma_{tR}(G)$-function.

Although the concept of total Roman domination number (as stated above) is very natural and several works has been already done on it, there is a fact which makes that its ``defensive'' features (according to the historical arising point of view) are notably increased to just achieve the ``total domination'' property. Consider for instance the graph drawn in Figure \ref{graph-example}.

\begin{figure}[ht]
\centering
\begin{tikzpicture}[scale=.6, transform shape]

\node [draw, shape=circle] (c) at  (0,0) {};
\node [draw, shape=circle] (a1) at  (2,1) {};
\node [draw, shape=circle] (a2) at  (4,1) {};
\node [draw, shape=circle] (b1) at  (2,3) {};
\node [draw, shape=circle] (b2) at  (4,3) {};
\node [draw, shape=circle] (a11) at  (2,-1) {};
\node [draw, shape=circle] (a21) at  (4,-1) {};
\node [draw, shape=circle] (b11) at  (2,-3) {};
\node [draw, shape=circle] (b21) at  (4,-3) {};
\node [draw, shape=circle] (r1) at  (-2,1) {};
\node [draw, shape=circle] (r2) at  (-2,-1) {};
\node [draw, shape=circle] (r3) at  (-2,3) {};
\node [draw, shape=circle] (r4) at  (-2,-3) {};

\node [scale=1.4] at (0,0.8) {\large 2};
\node [scale=1.4] at (2,1.7) {\large 1};
\node [scale=1.4] at (4,1.7) {\large 1};
\node [scale=1.4] at (2,3.7) {\large 1};
\node [scale=1.4] at (4,3.7) {\large 1};
\node [scale=1.4] at (2,-0.3) {\large 1};
\node [scale=1.4] at (4,-0.3) {\large 1};
\node [scale=1.4] at (2,-2.3) {\large 1};
\node [scale=1.4] at (4,-2.3) {\large 1};
\node [scale=1.4] at (-2.7,-3) {\large 0};
\node [scale=1.4] at (-2.7,3) {\large 0};
\node [scale=1.4] at (-2.7,-1) {\large 0};
\node [scale=1.4] at (-2.7,1) {\large 0};


\draw(a2)--(a1)--(c)--(b1)--(b2);
\draw(a21)--(a11)--(c)--(b11)--(b21);
\draw(r1)--(c)--(r2);
\draw(r3)--(c)--(r4);

\end{tikzpicture}
\hspace*{1cm}
\begin{tikzpicture}[scale=.6, transform shape]

\node [draw, shape=circle] (c) at  (0,0) {};
\node [draw, shape=circle] (a1) at  (2,1) {};
\node [draw, shape=circle] (a2) at  (4,1) {};
\node [draw, shape=circle] (b1) at  (2,3) {};
\node [draw, shape=circle] (b2) at  (4,3) {};
\node [draw, shape=circle] (a11) at  (2,-1) {};
\node [draw, shape=circle] (a21) at  (4,-1) {};
\node [draw, shape=circle] (b11) at  (2,-3) {};
\node [draw, shape=circle] (b21) at  (4,-3) {};
\node [draw, shape=circle] (r1) at  (-2,1) {};
\node [draw, shape=circle] (r2) at  (-2,-1) {};
\node [draw, shape=circle] (r3) at  (-2,3) {};
\node [draw, shape=circle] (r4) at  (-2,-3) {};

\node [scale=1.4] at (0,0.8) {\large 2};
\node [scale=1.4] at (2,1.7) {\large 2};
\node [scale=1.4] at (4,1.7) {\large 0};
\node [scale=1.4] at (2,3.7) {\large 2};
\node [scale=1.4] at (4,3.7) {\large 0};
\node [scale=1.4] at (2,-0.3) {\large 2};
\node [scale=1.4] at (4,-0.3) {\large 0};
\node [scale=1.4] at (2,-2.3) {\large 2};
\node [scale=1.4] at (4,-2.3) {\large 0};
\node [scale=1.4] at (-2.7,-3) {\large 0};
\node [scale=1.4] at (-2.7,3) {\large 0};
\node [scale=1.4] at (-2.7,-1) {\large 0};
\node [scale=1.4] at (-2.7,1) {\large 0};


\draw(a2)--(a1)--(c)--(b1)--(b2);
\draw(a21)--(a11)--(c)--(b11)--(b21);
\draw(r1)--(c)--(r2);
\draw(r3)--(c)--(r4);

\end{tikzpicture}
\caption{Two different labelings generating TRDFs on the drawn graph.}\label{graph-example}
\end{figure}
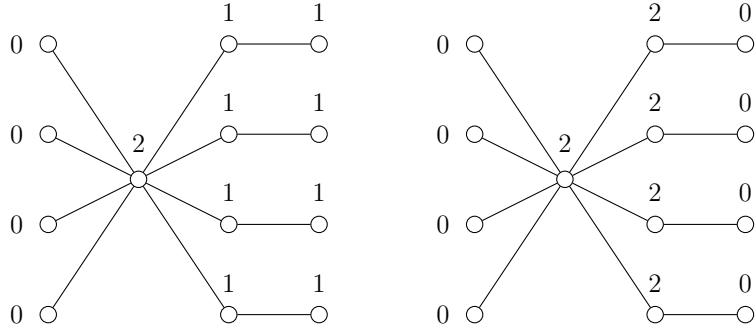

In the graph given in Figure \ref{graph-example}, the ``defensive'' features of both labelings are increased in order to have the ``total domination property'' of the function. That is, in one side, vertices labeled with one are not exactly ``in a dangerous'' situation since they have label one, and no other ``attacking'' vertex ``around'', or in second hand, there are two adjacent vertices labeled with two, and one of them has only one other neighbor. Thus, in some sense, one might consider allowing some not totally dominated vertices in a labeling whether the situation would be appropriate. For instance, a possible efficient (with respect to the sum of the weights assigned to the vertices) labeling having some partial total domination features of the graph of Figure \ref{graph-example} is shown in Figure \ref{graph-eff}. In connection with this situation, in this work we center our attention into some kind of a weaker version of the total Roman domination parameter, or a stronger version of the standard Roman domination one.

\begin{figure}[ht]
\centering
\begin{tikzpicture}[scale=.6, transform shape]

\node [draw, shape=circle] (c) at  (0,0) {};
\node [draw, shape=circle] (a1) at  (2,1) {};
\node [draw, shape=circle] (a2) at  (4,1) {};
\node [draw, shape=circle] (b1) at  (2,3) {};
\node [draw, shape=circle] (b2) at  (4,3) {};
\node [draw, shape=circle] (a11) at  (2,-1) {};
\node [draw, shape=circle] (a21) at  (4,-1) {};
\node [draw, shape=circle] (b11) at  (2,-3) {};
\node [draw, shape=circle] (b21) at  (4,-3) {};
\node [draw, shape=circle] (r1) at  (-2,1) {};
\node [draw, shape=circle] (r2) at  (-2,-1) {};
\node [draw, shape=circle] (r3) at  (-2,3) {};
\node [draw, shape=circle] (r4) at  (-2,-3) {};

\node [scale=1.4] at (0,0.8) {\large 2};
\node [scale=1.4] at (2,1.7) {\large 0};
\node [scale=1.4] at (4,1.7) {\large 1};
\node [scale=1.4] at (2,3.7) {\large 0};
\node [scale=1.4] at (4,3.7) {\large 1};
\node [scale=1.4] at (2,-0.3) {\large 0};
\node [scale=1.4] at (4,-0.3) {\large 1};
\node [scale=1.4] at (2,-2.3) {\large 0};
\node [scale=1.4] at (4,-2.3) {\large 1};
\node [scale=1.4] at (-2.7,-3) {\large 1};
\node [scale=1.4] at (-2.7,3) {\large 0};
\node [scale=1.4] at (-2.7,-1) {\large 0};
\node [scale=1.4] at (-2.7,1) {\large 0};


\draw(a2)--(a1)--(c)--(b1)--(b2);
\draw(a21)--(a11)--(c)--(b11)--(b21);
\draw(r1)--(c)--(r2);
\draw(r3)--(c)--(r4);

\end{tikzpicture}
\caption{An efficient labeling preserving some total dominating properties as well as allowing some other ones not totally dominated.}\label{graph-eff}
\end{figure}
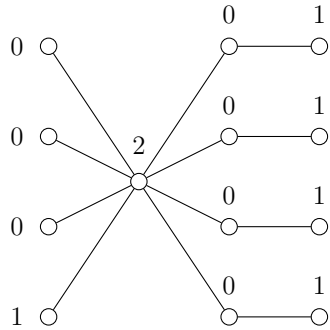

A \emph{quasi-total Roman dominating function} (QTRDF) on a graph $G$ is a function $f : V(G)\rightarrow \{0,1,2\}$ satisfying the following:
\begin{itemize}
  \item every vertex $u$ for which $f(u) = 0$ is adjacent to at least one vertex $v$ for which $f(v) =2$, and
  \item if $x$ is an isolated vertex in the subgraph induced by the set of vertices labeled with 1 and 2, then $f(x)=1$.
\end{itemize}
The \emph{weight} of a QTRDF $f$ is the value $\omega(f)=f(V(G))=\sum_{u\in V(G)} f(u)$. The minimum weight of a QTRDF on $G$ is called the \emph{quasi-total Roman domination number} of $G$ and is denoted by $\gamma_{qtR}(G)$.

In this work we introduce the concept above and begin with the study of its combinatorial and computational properties. In this sense, we next describe some basic terminology and notation that we shall need throughout our exposition and for the remaining basic terminology and notation on graph theory not given here, the reader is referred to \cite{w}.

We denote the vertex set and the edge set of a graph $G$ by $V(G)$ and $ E(G),$  respectively. The subgraph induced by $S \subseteq V(G)$ is denoted by $G[S]$. The complement of a graph $G$ is denoted by $\overline{G}$. For any vertex $x$ of a graph $G$, $N_G(x)$ denotes the set of all  neighbors of $x$ in $G$,
$N_G[x] = N_G(x) \cup \{x\}$ and the degree of $x$ is $d_G(x) = |N_G(x)|$. The  minimum and  maximum degree of a graph $G$ are denoted by $\delta(G)$ and $\Delta(G)$, respectively.
For a subset $A \subseteq V (G)$, let $N_G(A) = \cup_{x \in A} N_G(x)$ and $N_G[A] = N_G(A) \cup A$. For a graph $G$, let $x \in X \subseteq V(G)$. A vertex $y \in V(G)$ is a $X$-private neighbor of $x$ if $N_G[y] \cap X =\{x\}$. The $X$-private neighborhood  of $x$, denoted $pn_G(x,X)$, is the set of all $X$-private neighbors of $x$. The distance between two vertices $x, y \in V(G)$ is  denoted by $d_G(x, y)$.

If $G$ is a non-connected graph with connected components $C_1,\dots,C_s,I_1,\dots,I_r$, where each component $C_i$ has order at least two and each component $I_j$ is a singleton, then it is not difficult to observe that
$$\gamma_{qtR}(G)=r+\sum_{i=1}^s\gamma_{qtR}(G[C_i]).$$
In view of this fact, from now we only consider connected graphs, although we will not explicitly mention it, unless the state of the situation will require to involved not connected graphs.

\section{Primary combinatorial and computational results}

It is not difficult to see that any TRDF is also a QTRDF, as well as, a QTRDF is also a RDF. In concordance, the following chain of inequalities directly follows.
\begin{equation}\label{rd-wtrd-trd}
\gamma_R(G)\le \gamma_{qtR}(G)\le \gamma_{tR}(G)
\end{equation}

We first observe that the differences $\gamma_{qtR}(G)-\gamma_R(G)$ and $\gamma_{tR}(G)-\gamma_{qtR}(G)$ can be as large as possible. To this end, we consider the graph $G_1$, which is obtained as follows. We begin with a singleton vertex $u$. Next we add, $t$ copies of the graph $H$ drawn in Figure \ref{graph-H}, and join with an edge the vertex $u$ with the vertex $w$ of each added copy of $H$.

\begin{figure}[ht]
\centering
\begin{tikzpicture}[scale=.6, transform shape]

\node [draw, shape=circle] (x15) at  (0,10) {};
\node [draw, shape=circle] (x14) at  (0,8) {};
\node [draw, shape=circle] (x7) at  (-8,8) {};
\node [draw, shape=circle] (x25) at  (8,8) {};
\node [draw, shape=circle] (x1) at  (-12,6) {};
\node [draw, shape=circle] (x3) at  (-10,6) {};
\node [draw, shape=circle] (x6) at  (-8,6) {};
\node [draw, shape=circle] (x10) at  (-2,6) {};
\node [draw, shape=circle] (x13) at  (0,6) {};
\node [draw, shape=circle] (x19) at  (2,6) {};
\node [draw, shape=circle] (x24) at  (8,6) {};
\node [draw, shape=circle] (x29) at  (10,6) {};
\node [draw, shape=circle] (x34) at  (12,6) {};
\node [draw, shape=circle] (n1) at  (-6,6) {};

\node [scale=1.4] at (0,10.5) {\large $w$};
\node [scale=1.4] at (-8,8.6) {\large $w_1$};
\node [scale=1.4] at (8,8.6) {\large $w_3$};
\node [scale=1.4] at (0.8,8) {\large $w_2$};
\node [scale=1.4] at (-5.4,6) {\large $x$};
\node [scale=1.4] at (-8.7,6.1) {\large $x'$};

\draw(x1)--(x7)--(x15)--(x25)--(x34);
\draw(x3)--(x7)--(x6)--(n1);
\draw(x10)--(x14)--(x13);
\draw(x15)--(x14)--(x19);
\draw(x29)--(x25)--(x24);

\end{tikzpicture}
\caption{The graph $H$ with the labeling used in proofs.}\label{graph-H}
\end{figure}
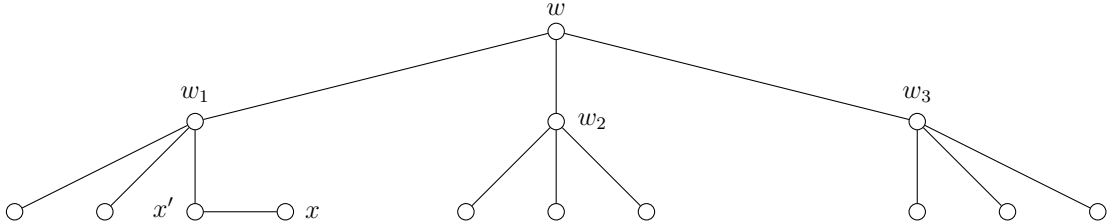

We consider the following functions defined on the graph $G_1$.
\begin{itemize}
  \item $f_1$: it assigns two to each vertex $w_i$, $i\in \{1,2,3\}$, of every copy of $H$, one to the vertex $x$ of every copy of $H$ and to the central vertex $u$, and zero otherwise.
  \item $f_2$: it assigns two to each vertex $w_i$, $i\in \{1,2,3\}$, of every copy of $H$, one to the vertices $x,x',w$ of every copy of $H$ and to the central vertex $u$, and zero otherwise.
  \item $f_3$: it assigns two to each vertex $w_i$, $i\in \{1,2,3\}$, of every copy of $H$, one to the vertices $x,w$ of every copy of $H$ and to the central vertex $u$, and zero otherwise.
\end{itemize}
Observe that $f_1$ is a RDF, $f_2$ is a TRDF and $f_3$ is a QTRDF, and so, $\gamma_R(G_1)\le \omega(f_1)=7t+1$, $\gamma_{tR}(G_1)\le \omega(f_2)=9t+1$ and $\gamma_{qtR}(G_1)\le \omega(f_3)=8t+1$. On the other hand, by making some observation on the construction, we deduce that such functions above are of minimum weight, which means we have equalities in these stated bounds. As a consequence, one can check that the differences $\gamma_{qtR}(G_1)-\gamma_R(G_1)=t$ and $\gamma_{tR}(G_1)-\gamma_{qtR}(G_1)=t$ can be as large as possible.

The inequalities given in (\ref{rd-wtrd-trd}) show lower and upper bounds for $\gamma_{qtR}(G)$ in terms of $\gamma_{R}(G)$ and $\gamma_{tR}(G)$, respectively. We next show that also lower and upper bounds for $\gamma_{qtR}(G)$ can be given, now in terms of $\gamma_{tR}(G)$ and $\gamma_{R}(G)$, respectively.

\begin{proposition}
Let $G$ be a graph without isolated vertices.
\begin{enumerate}[{\rm (i)}]
  \item If $f=(V_0,V_1,V_2)$ is any $\gamma_{R}(G)$-function, then $\gamma_{qtR}(G)\le \gamma_{R}(G)+|V_2|$.
  \item If $f'=(V'_0,V'_1,V'_2)$ is any $\gamma_{qtR}(G)$-function, then $\gamma_{tR}(G)\le \gamma_{qtR}(G)+|V'_1|$.
\end{enumerate}
\end{proposition}

\begin{proof}
Clearly, if we choose one neighbor with minimum weight of each vertex $v\in V_2$ and label it with one, then we obtain a QTRDF of $G$ with weight at most $\gamma_{R}(G)+|V_2|$. Similarly, if we take one neighbor with minimum weight of each vertex $v\in V'_1$ and label it with one, then we obtain a TRDF of $G$ with weight at most $\gamma_{qtR}(G)+|V'_1|$.
\end{proof}

We next give some families of graphs where the bounds above are achieved. Consider a graph $G$ without isolated vertices. We next add $k\ge 3$ pendant vertices to each vertex of $G$ and we also multisubdivide each edge of $G$ with two vertices  (by multisubdividing an edge $uv$ we mean removing the edge $uv$ and adding a path $P_2=xz$ and the edges $ux$, $zv$). We denote by $G_{2,k}$ the graph obtained in this way. Now, to obtain a graph $G'_k$ from $G$, we add $k\ge 3$ pendant vertices to each vertex of $G$ and subdivide (with only one vertex) exactly one the pendant edges of every vertex of $G$.

\begin{remark}
Let $G$ be any graph with minimum degree at least two.
\begin{enumerate}[{\rm (i)}]
  \item There exists a $\gamma_{R}(G_{2,k})$-function $f=(V_0,V_1,V_2)$ such that $\gamma_{qtR}(G_{2,k})=\gamma_{R}(G_{2,k})+|V_2|$.
  \item There exists a $\gamma_{qtR}(G'_k)$-function $f'=(V'_0,V'_1,V'_2)$ such that $\gamma_{tR}(G'_k)= \gamma_{qtR}(G'_k)+|V'_1|$.
\end{enumerate}
\end{remark}

\begin{proof}
(i) We first observe that $G_{2,k}$ has a unique $\gamma_{R}(G_{2,k})$-function $f=(V_0,V_1,V_2)$ of weight $2n$, where $n$ is the order of the graph $G$, such that $V_2=V(G)$, $V_1=\emptyset$ and $V_0=V(G_{2,k})\setminus V_2$. Clearly, such described function is a RDF on $G_{2,k}$, which means $\gamma_{R}(G_{2,k})\le 2n$. On the other hand, let $f'$ be any $\gamma_{R}(G_{2,k})$-function. If any pendant vertex $v$ of $G_{2,k}$ is labeled with one or two under $f'$, then the vertex $u\in V(G)$ adjacent to $v$ must have label 0 or 1, otherwise we can easily construct a RDF on $G_{2,k}$ with weight smaller than that of $f'$, which is not possible. Consequently, we consider all $k$ pendant vertices adjacent to $u$ have label one or two also. Since $k\ge 3$, by relabeling $u$ with two and its adjacent pendant vertices with zero, we obtain a new RDF with weight smaller than $\gamma_{R}(G_{2,k})$, a contradiction. Thus, all pendant vertices of $G_{2,k}$ are labeled zero under $f'$. As a consequence, every vertex $u\in V(G)$ must have label two, which leads to claim that all the vertices used in the multisubdivision process have label zero also. Since $f'$ is arbitrarily taken, because of its forced structure, we deduce that it must be unique and must have weight $2n$.

Now, by labeling with one, exactly one neighbor of each vertex $u\in V_2=V(G)$, from $f$ we obtain a QTRDF of weight $3n$ and so, $\gamma_{qtR}(G_{2,k})\le 3n=\gamma_{R}(G_{2,k})+|V_2|$. Suppose now that $\gamma_{qtR}(G_{2,k})<3n$ and let $g$ be any $\gamma_{qtR}(G_{2,k})$-function. Consider the partition of the vertex set of $G_{2,k}$ given by the closed neighborhoods of vertices of $G$, that is, the set $\{N[u]\,:\,u\in V(G)\}$. Since there are $n$ sets in this partition, there must exist a vertex $w\in V(G)$ such that $g(N[w])\le 2$. Since there are $k\ge 3$ pendant vertices in $N[w]$, at least one of them must have label zero, which means $g(w)=2$ and, in addition, all vertices in $N(w)$ must have label zero. Thus, $w$ is a vertex with label two under $g$ and has no neighbor with label one or two under $g$, which contradicts the fact that $g$ is a QTRDF. Therefore, $\gamma_{qtR}(G_{2,k})=3n=\gamma_{R}(G_{2,k})+|V_2|$.

(ii) By using somehow similar ideas to the proof of (i), we can deduce that $\gamma_{qtR}(G'_k)=3n$ with a unique $\gamma_{qtR}(G'_k)$-function $f'=(V'_0,V'_1,V'_2)$ given by $V'_2=V(G)$, $V'_1=\{v\in V(G'_k)\,:\,\mbox{$v$ has degree one and distance at least two to every vertex in $V(G)$}\}$ and $V'_0=V(G'_k)\setminus (V'_1\cup V'_2)$. Observe that the vertices labeled with one have no neighbors labeled with one or two. Thus, by labeling with one the unique neighbor of every vertex in $V'_1$, we obtain a TRDF of weight $4n$, and so, $\gamma_{tR}(G'_k)\le 4n=\gamma_{qtR}(G'_k)+|V'_1|$.

On the other hand, if we suppose that $\gamma_{tR}(G'_k)< 4n$, then again some relatively similar procedure as in item (i) will lead to a contradiction. Thus, $\gamma_{tR}(G'_k)=4n=\gamma_{qtR}(G'_k)+|V'_1|$, which completes the proof.
\end{proof}

It is natural to think that $\gamma_{qtR}(G)$ is related to other domination parameters of the graphs $G$. In this sense, we now present some bounds for $\gamma_{qtR}(G)$ in terms of $\gamma(G)$ and $\gamma_{t}(G)$, and for this, we need two known results from \cite{chellali2015} and \cite{total-Roman-1} and a new terminology. Given a set $S$ and a vertex $v$ of a graph $G$, the \emph{external private neighborhood} $epn_G(v,S)$ consists of those $S$-private neighbors of $v$ in $V(G)\setminus S$.

\begin{theorem}{\em \cite{chellali2015}}\label{total-Roman}
If $G$ is a graph without isolated vertices, then $\gamma_t(G)\leq \gamma_R(G)$.
\end{theorem}

\begin{theorem}{\em \cite{total-Roman-1}}\label{total-totalRoman}
If $G$ is a graph with no isolated vertex, then $\gamma_t(G) \le \gamma_{tR}(G) \le 2\gamma_t(G)$.
\end{theorem}

\begin{theorem}
For any nontrivial connected graph $G$, $$\gamma_t(G)\leq \gamma_{qtR}(G)\leq 2\gamma_t(G).$$
Furthermore,
\begin{enumerate}[{\rm (i)}]
\item $\gamma_{qtR}(G)=\gamma_t(G)$ if and only if $G\cong P_2$,
\item $\gamma_{qtR}(G)=\gamma_t(G)+1$ if and only if $\gamma_{qtR}(G)=3$, and
\item $\gamma_{qtR}(G)=2\gamma_t(G)$ if and only if $\gamma_{qtR}(G)=\gamma_{tR}(G) $ and $\gamma_{tR}(G)=2\gamma_t(G)$.
\end{enumerate}
\end{theorem}

\begin{proof}
By Theorem \ref{total-Roman} and inequality (\ref{rd-wtrd-trd}), we get $\gamma_t(G)\leq \gamma_{qtR}(G)$ and from inequality (\ref{rd-wtrd-trd}) and Theorem \ref{total-totalRoman}, we have $\gamma_{qtR}(G)\leq 2\gamma_t(G)$.

Now, let $f=(V_0,V_1,V_2)$ be a $\gamma_{qtR}(G)$-function such that $|V_2|$ is maximum. We consider $V_{1,2}=\{v\in V_1: N(v)\cap (V_1\cup V_2)\neq\emptyset\}$ and let $V_{1,0}=V_1\setminus V_{1,2}$. Notice that $epn(z, V_{1,0})\neq \emptyset$ for every $z\in V_{1,0}$. Otherwise, if there exists a vertex $z\in V_{1,0}$ such that $epn(z, V_{1,0})=\emptyset$, then there exist two vertices $y\in V_{1,0}$ and $x\in V_0$ such that $x\in N(z)\cap N(y)$. Notice that $x$ is adjacent to a vertex of $V_2$. Hence, the function $f'=(V_0',V_1',V_2')$, defined by $f'(x)=2$, $f'(y)=f'(z)=0$ and $f'(u)=f(u)$ if $u\in V(G)\setminus \{x,y,z\}$, is a QTRDF on $G$ satisfying $|V_2'|>|V_2|$, which is a contradiction.

Now, let $A$ be a subset of $\cup_{z\in V_{1,0}}epn(z, V_{1,0})$ such that $A$ contains exactly one vertex of $epn(z, V_{1,0})$ for every $z\in V_{1,0}$.  Observe that $|A|=|V_{1,0}|$, and that $A\cup V_{1,2}\cup V_2$ is a total dominating set of $G$. Thus
$$\gamma_t(G)\leq |A|+|V_{1,2}|+|V_2|=|V_1|+|V_2|\leq |V_1|+2|V_2|=\gamma_{qtR}(G).$$

(i) If $\gamma_{qtR}(G)=\gamma_t(G)$, then $V_2=\emptyset$, which implies that $V_0=\emptyset$ and $\gamma_{qtR}(G)=\gamma_t(G)=n$. Therefore $G\cong P_2$. The other implication is straightforward to see.

(ii) If $\gamma_{qtR}(G)=\gamma_t(G)+1$, then $|V_2|=1$ and $|A|=|V_{1,0}|=0$. Hence $\gamma_{qtR}(G)=3$.  The other implication is straightforward to see.

(iii) The proof of this item is straightforward to observe, and therefore, omitted.
\end{proof}

Since $\gamma_t(G)\le 2\gamma(G)$ for any graph $G$ (see \cite{HeYe_book}), the theorem above leads to $\gamma_{qtR}(G)\leq 4\gamma(G).$ However, this bound can be improved as we next show.

\begin{theorem}
Let $G$ be any graph without isolated vertices. If $f=(V_0,V_1,V_2)$ is a $\gamma_{qtR}(G)$-function and $V_{1,2}^*=\{v\in V_1\,:\,N(v)\cap V_2\ne \emptyset\}$, then $$\gamma(G)+|V_2|+|V_{1,2}^*|\le \gamma_{qtR}(G)\le 3\gamma(G).$$
\end{theorem}

\begin{proof}
We consider a dominating set $S$. Now, we construct a set $S'$ in the following way. For any vertex $v\in S$ such that $N(v)\cap S=\emptyset$, we choose one neighbor $v'$ of $v$ (which exists since $G$ has no isolated vertices), and add it to $S'$. Now, the function $f=(V(G)\setminus(S\cup S'),S',S)$ is a QTRDF on $G$ of weight at most $3\gamma(G)$, and we are done for the upper bound.

Now, we prove the lower bound. Observe that $V_2\cup (V_1\setminus V_{1,2}^*)$ is a dominating set in $G$. Thus,
$$\gamma(G)\le |V_2|+|V_1\setminus V_{1,2}^*|=|V_2|+|V_1|-|V_{1,2}^*|=2|V_2|+|V_1|-|V_2|-|V_{1,2}^*|=\gamma_{qtR}(G)-|V_2|-|V_{1,2}^*|,$$
which completes the proof.
\end{proof}

For the tightness of the upper bound above we consider for instance any star graph, or any path or cycle graph of order $3k$ for some integer $k\ge 1$. \\

As we can may guess, the problem of computing the quasi-total Roman domination number of graphs is as hard as its related predecessors. The NP-completeness of the decision problem concerning $\gamma_R(G)$ was proved in \cite{Dreyer}, while a similar result was obtained in \cite{lc} for the case of $\gamma_{tR}(G)$. At next we prove also the NP-completeness of the decision problem related to $\gamma_{qtR}(G)$, that is, of the next problem.

$$\begin{tabular}{|l|}
  \hline
  \mbox{QUASI-TOTAL ROMAN DOMINATION problem (QTRD problem)}\\
  \mbox{INSTANCE: A non trivial graph $G$ without isolated vertices}
  \mbox{and a positive integer $r$.}\\
  \mbox{PROBLEM: Deciding whether $\gamma_{qtR}(G)$ is less than $r$.}\\
  \hline
\end{tabular}$$

In order to study the problem above from a complexity point of view, we shall make a reduction from the NP-complete decision problem concerning computing the Roman domination number of graphs, \emph{i.e.}, the next problem.

$$\begin{tabular}{|l|}
  \hline
  \mbox{ROMAN DOMINATION problem}\\
  \mbox{INSTANCE: A non trivial graph $G$ and a positive integer $r$.}\\
  \mbox{PROBLEM: Deciding whether the Roman domination number of $G$}
  \mbox{is larger than $r$.}\\
  \hline
\end{tabular}$$

We next introduce a graph $G'$ to be used in the reduction. Consider a graph $H$ as shown in Figure \ref{graph-H}. Now, given any graph $G$ of order $n$, a graph $G'$ is constructed as follows. For any vertex $z\in V(G)$, we add a copy of the graph $H$ and the edge $zw$. By using this construction and the ROMAN DOMINATION problem, we are able to prove the NP-completeness of the QTRD problem.

\begin{theorem}
The QTRD problem is NP-complete.
\end{theorem}

\begin{proof}
Consider any graph $G$ of order $n$, and the graph $G'$ as described above. Let $f$ be any $\gamma_R(G)$-function, and let $f'$ be a function on $G'$ defined as follows. For any vertex $v\in V(G)$, $f'(v)=f(v)$. For any vertex $v\in V(H)$ of any copy of $H$ used to generate $G'$, if $v=w$ or $v=x$, then $f(v)=1$; if $v=w_i$, for some $i\in\{1,2,3\}$, then $f(w_i)=2$; and $f(v)=0$ otherwise. We observe that every vertex labeled with zero is adjacent to a vertex labeled with two, and that any vertex labeled with two has one neighbor labeled with one or two. Thus, $f'$ is a QTRDF on $G'$, which leads to $\gamma_{qtR}(G')\le \omega(f')=8n+\gamma_R(G)$.

On the other hand, assume $g=(V_0,V_1,V_2)$ is a $\gamma_{qtR}(G')$-function. According to the structure of $G'$, we can readily seen that $g(w_i)=2$ for every $w_i\in V(H)$ of any copy of $H$, used to generate $G'$. In order to have the partial total domination property in the set $V_1\cup V_2$, it must happen that each vertex $w_i$ labeled two will have a neighbor labeled one or two. In this sense, it must also be $g(w)=1$ for any copy of $H$. Moreover, $g(x)\ge 1$ also occurs. As a consequence, any vertex $u\in V(G)$ for which $g(u)=0$ is adjacent to a vertex $u'\in V(G)$ for which $g(u')=2$, which means that the function $g$ restricted to $V(G)$ is a RDF on $G$, and so, $\gamma_R(G)\le g(V(G))$. Thus, $\gamma_{qtR}(G')=\omega(g)=g(V(G))+\sum_{i=1}^{n}g(V(H))\ge\gamma_R(G)+ \sum_{i=1}^{n}8=8n+\gamma_R(G)$, and therefore, the equality $\gamma_{qtR}(G')=8n+\gamma_R(G)$ follows. Now, by taking $j=8n+k$, it is readily seen that $\gamma_{qtR}(G')\leq j$ if and only if $\gamma_R(G)\leq k$, which completes the reduction.
\end{proof}

According to the result above, it clearly happens that computing the quasi-total Roman domination number of graphs is an NP-hard problem.
In this sense, it is desirable to bound its value or compute the exact value for several non trivial families of graphs. We next center our attention on these goals. To this end, one first consider some algorithmic technique for finding a QTRDF of a graph. Despite it does not behave so ``efficiently'', it has a positive value while giving some bounds or proving some results concerning the quasi-total Roman domination number of graphs.

\begin{algorithm}[ht]
\caption{Finding a QTRDF of a graph}\label{algor}
\vspace*{0.2cm}
Input: A graph $G$ of order $n\geq 2$\\
Output: a QTRDF
\begin{algorithmic}
\STATE $V_0=\emptyset$, $V_1=\emptyset$, $V_2=\emptyset$
\STATE order $V(G)$ with respect to the degree of the vertices in decreasing order
\STATE $L$ is the list of ordered vertices
\WHILE{$|L|\ge 1$ and the first element of $L$ has degree at least one}
\STATE take first vertex $v\in L$
\STATE add $v$ to $V_2$
\STATE add a neighbor $v'$ of $v$ to $V_1$
\STATE add $N(v)-\{v'\}$ to $V_0$
\STATE remove $N[v]$ from $L$
\STATE update $G$ as the graph induced by $V(G)-N[v]$
\STATE order $V(G)$ with respect to the degree of the vertices in decreasing order
\STATE update $L$
\ENDWHILE
\STATE add remaining vertices to $V_1$ (vertices with degree zero)
\end{algorithmic}
\end{algorithm}

In each step of Algorithm \ref{algor}, the contribution of the weights assigned to the vertices over the weight of the constructed function $f$ is three. Moreover, the algorithm creates a partition $\Pi=\{A_1,A_2,\dots,A_q,I\}$ of the vertex set $V(G)$ such that $A_i=N[v_i]$, for some vertex $v_i\in V(G)$ and $|A_i|\ge 2$ for every $i\in\{1,\dots,q\}$; and $I$ induces a set of isolated vertices. In consequence, $\gamma_{qtR}(G)\le \omega(f)=3q+|I|$.
We next find some specific cases where such bound gives some interesting consequences.

\begin{proposition}\label{Delta-order}
If a graph $G$ of order $n$ has $q\ge 1$ vertices $v_1,v_2,\dots, v_q$ of degree larger than two such that the distance between any two of them is at least three, and the set of vertices not adjacent to a vertex $v_i$, $i\in \{1,\dots,q\}$, form an independent set, then $\gamma_{qtR}(G)\le n+3q-\sum_{i=1}^q(d_G(v_i)+1)$.
\end{proposition}

\begin{proof}
By running Algorithm \ref{algor} over such graph, we notice the following facts.
\begin{itemize}
  \item The algorithm makes $q$ steps.
  \item In each step, the closed neighborhood of a vertex of degree larger than two is removed.
  \item At the end of a running time, a set of isolated vertices $I$ remains (these ones not adjacent to a vertex of degree larger than two).
\end{itemize}
In consequence, we deduce that the weight of the function created by the algorithm has weight  $3q+|I|$. Now, since $|I|=n-\sum_{i=1}^q(d_G(v_i)+1)$, we deduce the upper bound.
\end{proof}

A trivial example which shows the tightness of the bound above can be easily noticed by taking any graph $G$ of order $n$ and maximum degree $n-1$. In this case, we obtain $\gamma_{qtR}(G)\le 3$ which is exact as we know from Proposition \ref{3-and-n} (i). We next present another
family of graphs for which the bound above is tight.

Similarly to the graph $G_{2,k}$ previously defined, we define the graph $G_{3,k}$ as a graph obtained from a graph $G$, but now the multisubdivision process is made with 3 vertices. In this sense, we observe that $G_{3,k}$ has order $n+3m+nk=n(k+1)+3m$, where $n$ and $m$ are the order and size of the graph $G$ used to generate $G_{3,k}$.

\begin{remark}
If $G$ is a graph of minimum degree at least three, order $n$ and size $m$, then $\gamma_{qtR}(G_{3,k})=3n+m$.
\end{remark}

\begin{proof}
Since $G_{3,k}$ has order $n+3m+nk=n(k+1)+3m$, the vertices corresponding to that ones of the original graph $G$ has degree at least three (and can be taken as the set of $q$ vertices of Proposition \ref{Delta-order}), and those vertices not adjacent to the vertices of $G$ form an independent set in $G_{3,k}$, from Proposition \ref{Delta-order} we have that
$$\gamma_{qtR}(G_{3,k})\le n+3m+nk+3n-\sum_{v\in V(G)}(d_{G_{3,k}}(v)+1)= n+3m+nk+3n-2m-n-nk =3n+m.$$
On the other hand, let $f$ be any $\gamma_{qtR}(G_{3,k})$-function such the set of vertices labeled with two is of minimum cardinality. If there exists a vertex $v\in V(G)$ for which $f(v)\ne 2$, then every pendant vertex adjacent to $v$ must have label at least one. Since there are at least three pendant vertices adjacent to $v$, by some simple case analysis, we observe that one can always find a QTRDF of weight smaller than that of $f$, which is a contradiction. Thus, every vertex $v\in V(G)$ satisfies $f(v)=2$. Since every vertex labeled with two needs a neighbor with label one or two, there should be at least one neighbor $v'\in N(v)$ for every $v\in V(G)$ such that $f(v')\ge 1$. Thus, since the distance in $G_{3,k}$ between any two vertices of $G$ is at least four, it follows $f(N_{G_{3,k}}[v])\ge 3$ for every vertex $v\in V(G)$, which leads to $f(V(G_{3,k})\setminus \bigcup_{v\in V(G)}N_{G_{3,k}}[v])=f(V(G_{3,k}))-f(\bigcup_{v\in V(G)}N_{G_{3,k}}[v])\le \gamma_{qtR}(G_{3,k})-3n\le m$.

If $f(V(G_{3,k})\setminus \bigcup_{v\in V(G)}N_{G_{3,k}}[v])<m$, then since $|V(G_{3,k})\setminus\bigcup_{v\in V(G)}N_{G_{3,k}}[v]|=m$, there exists a vertex (a central vertex of one path used to multisubdivide one edge of $G$) $x$ such that $f(x)=0$. This means that one of its two neighbors, say $x'$ satisfies $f(x')=2$. Now, notice that $x'\in N(u)$ for some $u\in V(G)$, and that $f(u)=2$. Now, by relabeling the vertices $x$ and $x'$ with one, we obtain a new $\gamma_{qtR}(G_{3,k})$-function with a smaller number of vertices labeled with two, which is a contradiction with our assumption. Consequently, $f(V(G_{3,k})\setminus \bigcup_{v\in V(G)}N_{G_{3,k}}[v])=m$, and so
$$\gamma_{qtR}(G_{3,k})=f(V(G_{3,k})\setminus\bigcup_{v\in V(G)}N_{G_{3,k}}[v])+f(\bigcup_{v\in V(G)}N_{G_{3,k}}[v])\ge m+3n,$$
and the equality follows.
\end{proof}

We next continue with some other bounds for $\gamma_{qtR}(G)$ in terms of other parameters.

\begin{theorem}\label{maximum-degree}
For any graph $G$ of order $n$ and maximum degree $\Delta(G)\ge 2$,
$$\left\lceil \frac{2n}{\Delta}\right\rceil\le \gamma_{qtR}(G)\le n-\Delta(G)+2.$$
\end{theorem}

\begin{proof}
We first prove the lower bound. Let $f=(V_0,V_1,V_2)$ be a $\gamma_{qtR}(G)$-function. By definition, each vertex $v\in V_0$ is adjacent to at least one vertex $u\in V_2$, and since every vertex in $V_2$ has a neighbor labeled with positive weight under $f$, we must have $|V_0|\le (\Delta-1)|V_2|$. So, it follows
\begin{eqnarray*}
\Delta \gamma_{qtR}(G)&=&\Delta(|V_1|+2|V_2|)\\
                       &\geq &\Delta|V_1|+2|V_0|+2|V_2|\\
                       &\geq &2|V_1|+2|V_0|+2|V_2|\\
                       &\geq & 2n.
\end{eqnarray*}
Therefore, $\gamma_{qtR}(G)\geq \left\lceil \frac{2n}{\Delta}\right\rceil$.

To prove the upper bound, we observe that, if $v$ is a vertex of maximum degree in $G$ and $v'$ is a neighbor of $v$, then the function $g=(N(v)\setminus \{v'\},(V(G)\setminus N[v])\cup \{v'\},\{v\})$ is a QTRDF on $G$ with weight $n-\Delta(G)+2$.
\end{proof}

For the particular case in which $\Delta(G)=2$, Theorem \ref{maximum-degree} leads to $\gamma_{qtR}(G)= n$, and we next prove that the only case in which $\gamma_{qtR}(G)=n$ is precisely whether $\Delta(G)=2$. Moreover, since every vertex labeled with two in a QTRD must have a neighbor labeled one or two, we can readily seen that $\gamma_{qtR}(G)\ge 3$ for any graph $G$, unless $G$ is the graph $K_2$.

\begin{proposition}\label{3-and-n}
For any graph $G$ of order $n\ge 3$, $3\le \gamma_{qtR}(G)\le n$. Moreover,
\begin{enumerate}[{\rm (i)}]
  \item $\gamma_{qtR}(G)=3$ if and only if $G$ is a graph of maximum degree $n-1$,
  \item $\gamma_{qtR}(G)=4$ if and only if $\gamma(G)=\gamma_t(G)=2$, and
  \item $\gamma_{qtR}(G)=n$ if and only if $G$ is a path or a cycle of order at least three.
\end{enumerate}
\end{proposition}

\begin{proof}
(i) If $G$ is a graph of maximum degree $n-1$, then we can easily see that $\gamma_{qtR}(G)=3$. On the contrary, assume $\gamma_{qtR}(G)=3$ and let $g=(V_0,V_1,V_2)$ be a $\gamma_{qtR}(G)$-function. Hence, it must happen one of the following situations.
\begin{itemize}
  \item $|V_1|=3$ and $|V_2|=0$. Thus, $V_0=\emptyset$, which means $n=3$ and so $G$ is a path $P_3$ or a cycle $C_3$, and both of them have degree $n-1$.
  \item $|V_1|=1$ and $|V_2|=1$. Thus, the two vertices in $V_1\cup V_2$ are adjacent, and every other vertex not in $V_1\cup V_2$ has label zero, and so it is adjacent to the vertex in $V_2$. Therefore, the vertex in $V_2$ has degree $n-1$.
\end{itemize}

(ii) First, we assume that $\gamma_{qtR}(G)=4$ and let $f=(V_0,V_1,V_2)$ be a $\gamma_{qtR}(G)$-function. Thus, one of the following situations occurs.

\begin{itemize}
  \item $|V_1|=4$ and $|V_2|=0$. In this case, $V_0=\emptyset$ and so, $G$ is a graph of order $n=4$, and by item (i), $G$ is the path $P_4$ or the cycle $C_4$, which satisfy that $\gamma(P_4)=\gamma_t(P_4)=\gamma(C_4)=\gamma_t(C_4)=2$.

  \item $|V_1|=2$ and $|V_2|=1$. Hence, exactly one vertex of the two vertices in the set $V_1$ is adjacent to the vertex in $V_2$, and every other vertex $x\notin V_1\cup V_2$ ($x\in V_0$) is adjacent to the vertex in $V_2$. Thus, the vertex in $V_2$, say $v$, together with a neighbor of the vertex in $V_1$ not adjacent to $v$ form a total dominating set, which means $\gamma_t(G)=2$. Also, since $\gamma_{qtR}(G)>3$, we have that $\gamma(G)\geq 2$. Therefore, $\gamma(G)=\gamma_t(G)=2$.

  \item $|V_1|=0$ and $|V_2|=2$. Clearly, every vertex in $V(G)\setminus V_2$ is adjacent to a vertex of $V_2$, and the two vertices in $V_2$ must be adjacent. Thus, $V_2$ is a total dominating set, which means $\gamma_t(G)=2$. Since $\gamma_{qtR}(G)>3$, we have that $\gamma(G)\geq 2$. Therefore, $\gamma(G)=\gamma_t(G)=2$.

\end{itemize}

On the other hand, let $\gamma(G)=\gamma_t(G)=2$. Notice that $G$ must have maximum degree at most $n-2$. Hence, by item (i), we have that $\gamma_{qtR}(G)\geq 4$. Notice that $\gamma_{qtR}(G)\leq \gamma_{tR}(G)\leq 2\gamma_t(G)=4$ (which is known from \cite{total-Roman-1}). Therefore, $\gamma_{qtR}(G)=4$, which completes the proof of (ii).

(iii) Assume $G$ is a path or a cycle of order at least three.
 By Theorem \ref{maximum-degree} we obtain the equality. On the other hand, let $G$ be a graph of order $n\geq 3$ satisfying $\gamma_{qtR}(G)=n$. If $G$ has maximum degree larger than two, then by Theorem \ref{maximum-degree}, we obtain $\gamma_{qtR}(G)\le n-1$, a contradiction. Thus, $G$ has maximum degree two, which means $G$ is a path or a cycle of order at least three.
\end{proof}

Another bound for the quasi-total Roman domination number of a graph $G$ can be given in terms of the packing number of $G$. To this end, it is said that a subset $B\subseteq V(G)$ is a {\em packing} in $G$ if for every pair of vertices $u,v\in B$, $N[u]\cap N[v]=\emptyset$. The {\em packing number} $\rho(G)$ is the maximum cardinality of a packing in $G$.

\begin{proposition}\label{packings}
For any graph $G$ of order $n$, $\gamma_{qtR}(G)\le n-\rho(G)(\delta(G)-2)$.
\end{proposition}

\begin{proof}
Let $S$ be a packing in $G$. We now consider a function $f=(V_0,V_1,V_2)$ such that $V_2=S$, $V_1$ is formed by exactly one neighbor of each vertex of $S$ together with that vertices not adjacent to any vertex in $S$, and $V_0$ the remaining vertices. We can easily observe that $f$ is a QTRDF on $G$. Since $f$ has weight $3\rho(G)+n-\sum_{v\in S}(d_G(v)+1)$, we directly deduce the bound.
\end{proof}

Clearly, the bound above is only useful whether $G$ has degree at least two. Note that precisely, if $G$ has minimum degree two, then we get $\gamma_{qtR}(G)\le n$, which is tight for the case of cycles. Moreover, if $G$ is a complete graph of order at least three, then $\rho(G)=1$ and $\delta(G)=n-1$, which leads to $\gamma_{qtR}(G)\le 3$, that is also a sharp value.

An special class of graphs $G$ for which its packing number equals its domination number is known as efficient domination graphs. Notice that in such case, the closed neighborhood centered in vertices belonging to a $\rho(G)$-set partitions the vertex set of $G$. Thus, the next result can be deduced from the proof of Proposition \ref{packings}.

\begin{proposition}\label{packings-1}
For any efficient domination graph $G$, $\gamma_{qtR}(G)\le 3\rho(G)$.
\end{proposition}

\section{A Nordhaus-Gaddum bound}

In 1956 (see \cite{ng}), Nordhaus and Gaddum began some research idea, which has further became very useful and popular in the investigation concerning graph theory, and probably more intensively, in the area of domination in graphs. They gave lower and upper bounds on the sum and product of the chromatic numbers of a graph and its complement in terms of the order of the graph. From this seminal work, several results concerning $\Psi(G)+\Psi(\overline{G})$ or $\Psi(G)\Psi(\overline{G})$ for some given parameter $\Psi$ of graphs have been appeared. In consequence, searching Nordhaus-Gaddum type inequalities has centered the attention of a large number of investigations. For interested readers, we suggest the fairly complete survey \cite{ah}. For the specific case of Roman domination parameters, we suggest for example the works \cite{amjadi,chambers}.

By using Theorem \ref{maximum-degree}, we can deduce a Nordhaus-Gaddum bound for the quasi-total Roman domination number (only for the sum version). To this end, we first notice the following. Since it cannot happen that a graph $G$ of order $n$, and its complement $\overline{G}$, will both have maximum degree $n-1$, the sum $\gamma_{qtR}(G)+\gamma_{qtR}(\overline{G})$ will be equal to six only whether $G$ has order three. Note that by Proposition \ref{3-and-n}, $\gamma_{qtR}(G)\ge 3$ for any graph $G$ of order $n\ge 3$, and the equality is attained only for graphs of maximum degree $n-1$. In consequence, in our next result we just consider graphs of order at least four. Moreover, not connected graphs are considered (in contrast with our comment from the introductory section). From now on, we use the following families of graphs.
\begin{itemize}
  \item $\mathcal{F}_1$: All the graphs of order $n$ and maximum degree $n-1$, that have exactly one vertex of degree $n-1$ and at least one vertex of degree one.
  \item $\mathcal{F}'_1$: All the graphs which are obtained as the complement of a graph in $\mathcal{F}_1$.
\end{itemize}
We also remark that in this section we involve the study of not connected graphs as well.

\begin{theorem}
If $G$ is a graph of order $n\ge 4$, then $7\le \gamma_{qtR}(G)+\gamma_{qtR}(\overline{G})\le n+5$.
Moreover,
\begin{itemize}
  \item[{\rm (i)}] $\gamma_{qtR}(G)+\gamma_{qtR}(\overline{G})=7$ if and only if $G,\overline{G}\in \{K_4, \overline{K_4}, K_4-e, \overline{K_4-e},\mathcal{F}_1,\mathcal{F}'_1\}$, and
  \item[{\rm (ii)}] $\gamma_{qtR}(G)+\gamma_{qtR}(\overline{G})=n+5$ if and only if $G$ is $C_5$.
\end{itemize}
\end{theorem}

\begin{proof}
Since $\gamma_{qtR}(G)\ge 3$ for any graph $G$ of order at least three and it cannot happen that $G$ and its complement $\overline{G}$, will both have maximum degree $n-1$, by using Proposition \ref{3-and-n}, the lower bound trivially follows. On the other hand, if $G$ has maximum degree zero, then $\overline{G}$ is a complete graph, and so $\gamma_{qtR}(G)+\gamma_{qtR}(\overline{G})=n+3<n+5$. If $G$ has maximum degree one (which means $\gamma_{qtR}(G)=n$), then $\overline{G}$ has maximum degree at least $n-2$. In such case, we can readily observe that $\gamma_{qtR}(\overline{G})\le 4$. Thus $\gamma_{qtR}(G)+\gamma_{qtR}(\overline{G})\le n+4<n+5$. Since $G$ and $\overline{G}$ can be exchanged without loss of generality, from now on we may assume that both $G$ and $\overline{G}$ have maximum degree at least two. Hence, from Theorem \ref{maximum-degree}, we have that
\begin{equation}\label{nord-has}
\gamma_{qtR}(G)+\gamma_{qtR}(\overline{G})\le n-\Delta(G)+2+n-\Delta(\overline{G})+2=n-\Delta(G)+\delta(G)+5\le n+5,
\end{equation}
and the proof of the upper bound is completed.

(i) If $G,\overline{G}\in \{K_4, \overline{K_4}, K_4-e, \overline{K_4-e},\mathcal{F}_1,\mathcal{F}'_1\}$, then it is straightforward to observe that $\gamma_{qtR}(G)+\gamma_{qtR}(\overline{G})=7$.

On the contrary, we assume now that $\gamma_{qtR}(G)+\gamma_{qtR}(\overline{G})=7$. If $G$ has maximum degree zero, then $\overline{G}$ is a complete graph, which means $7=\gamma_{qtR}(G)+\gamma_{qtR}(\overline{G})=n+3$, and so, $n=4$ (or equivalently $G,\overline{G}\in \{K_4, \overline{K_4}\}$). If $G$ has maximum degree one, then its complement has maximum degree at least $n-2$. If $\overline{G}$ has maximum degree $n-2$, then $7=\gamma_{qtR}(G)+\gamma_{qtR}(\overline{G})=n+4$, which means $n=3$ and this is not possible. If $\overline{G}$ has maximum degree $n-1$, then $7=\gamma_{qtR}(G)+\gamma_{qtR}(\overline{G})=n+3$, which means $n=4$ (or equivalently $G,\overline{G}\in \{K_4, \overline{K_4}\}$). Hence, without loss of generality, we may assume that both $G$ and $\overline{G}$ have maximum degree at least two.

If $G$ has maximum degree $n-1$, then $\gamma_{qtR}(G)=3$. Thus, $7=\gamma_{qtR}(G)+\gamma_{qtR}(\overline{G})=\gamma_{qtR}(\overline{G})+3$, and so $\gamma_{qtR}(\overline{G})=4$. Since $\overline{G}$ has at least one isolated vertex (a vertex of degree $n-1$ in $G$), say $u$, it must happen that $V(\overline{G})\setminus\{u\}$ induces a graph of maximum degree $n-2$, which means there must be a vertex of degree one in $G$ (note that $u$ must have label one under any QTRDF on $\overline{G}$ and so, the quasi-total Roman domination number of the graph induced by $V(\overline{G})\setminus\{u\}$ is three). If $G$ has only one vertex of degree $n-1$, then $G\in \mathcal{F}_1$. If $G$ has exactly two vertices, say $x,y$, of degree $n-1$, then again $\gamma_{qtR}(\overline{G})=4$ and $x,y$ are isolated vertices in $\overline{G}$, which means they must have label one under any QTRDF on $\overline{G}$. Thus, the subgraph induced by $V(\overline{G})\setminus\{x,y\}$ must have only two vertices, and they should be adjacent, since $G$ has only two vertices of degree $n-1$. Consequently, $G$ is a complete graph of order four minus one edge, and $\overline{G}$ (its complement), is given by the union of two isolated vertices and $P_2$. If $G$ has three vertices, say $x_1,x_2,x_3$, of degree $n-1$, then as before $\gamma_{qtR}(\overline{G})=4$ and $x_1,x_2,x_3$ are isolated vertices in $\overline{G}$, and they should have label one under any QTRDF on $\overline{G}$. Thus, we can deduce that $V(\overline{G})\setminus\{x_1,x_2,x_3\}$ has only one vertex, and consequently, $G$ is $K_4$ and $\overline{G}$ is an edgeless graph of order four. If $G$ has more than three vertices of degree $n-1$, then we obtain a contradiction.

Finally, without loss of generality we may assume that both $G$ and $\overline{G}$ have maximum degree smaller than $n-1$. But then, by using Proposition \ref{3-and-n}, we obtain that $\gamma_{qtR}(G)+\gamma_{qtR}(\overline{G})\ge 8$, which is not possible, and this completes the proof of (i).\\

(ii) It is clear that $\gamma_{qtR}(C_5)+\gamma_{qtR}(\overline{C_5})=10=n+5$. On the other hand, assume that $\gamma_{qtR}(G)+\gamma_{qtR}(\overline{G})=n+5$. Hence, there should be equalities in the chain of inequalities (\ref{nord-has}), which leads to $\Delta(G)=\delta(G)$, or equivalently that $G$ is a $k$-regular graph for some integer $k$. Also, without loss of generality, we may assume that $k\le (n - 1)/2$. If $k=0$, then $\gamma_{qtR}(G)=n$ and $\gamma_{qtR}(\overline{G})=3$ ($\overline{G}$ is a complete graph), and so $\gamma_{qtR}(G)+\gamma_{qtR}(\overline{G})=n+3$, a contradiction. If $k=1$, then we can readily see that $\gamma_{qtR}(G)=n$ and $\gamma_{qtR}(\overline{G})=4$ ($\overline{G}$ is an $(n-2)$-regular graph). Thus, $\gamma_{qtR}(G)+\gamma_{qtR}(\overline{G})=n+4$,  which is again a contradiction. If $k=2$, then $G$ is a cycle, and by Proposition \ref{3-and-n} (iii), $\gamma_{qtR}(G)=n$. If $G$ is $C_4$, then $\gamma_{qtR}(G)+\gamma_{qtR}(\overline{G})=7<n+5$, which is not possible. If $G$ is $C_5$, then $\overline{G}$ is also $C_5$, and we get the equality $\gamma_{qtR}(G)+\gamma_{qtR}(\overline{G})=10=n+5$. If $n\ge 6$, then $\overline{G}$ always contains two adjacent vertices that dominates the vertex set of $\overline{G}$. Consequently, $\gamma_{qtR}(\overline{G})\le 4$, which leads to $\gamma_{qtR}(G)+\gamma_{qtR}(\overline{G})\le n+4$, by also using Proposition \ref{3-and-n} (iii), and this is a contradiction.

From now on we assume $k\ge 3$. We now consider $G$ has diameter at least three. In this sense, there are two vertices, say $z_1,z_2$, such that $N[z_1]\cap N[z_2]=\emptyset$. Let $z'_1\in N(z_1)$ and $z'_2\in N(z_2)$. Hence, we observe that the function $f=(N(\{z_1,z_2\})\setminus\{z'_1,z'_2\},(V(G)\setminus N[\{z_1,z_2\}])\cup\{z'_1,z'_2\},\{z_1,z_2\})$ is a QTRDF of weight $n-2(k+1)+6=n-2k+4$. Thus $\gamma_{qtR}(G)\le n-2k+4$. On the other hand, since $G$ has diameter at least three, there are two adjacent vertices in $\overline{G}$ that dominates $\overline{G}$, and this means $\gamma_{qtR}(\overline{G})\le 4$. As a consequence, $\gamma_{qtR}(G)+\gamma_{qtR}(\overline{G})\le n-2k+8$, and since $k\ge 3$, we obtain that $\gamma_{qtR}(G)+\gamma_{qtR}(\overline{G})\le n+2$, which is not possible.

Accordingly, it must happen $G$ has diameter two, and by the symmetry of the proving process, we may assume also $\overline{G}$ has diameter two. Let $a$ be any vertex of $G$. We now consider any vertex $b$ at distance two from $a$ (such vertex exists because $G$ has maximum degree smaller than $n-1$), and let $r$ be the number of neighbors in common between $a$ and $b$. Clearly $1\le r\le k$. Let $R=N(a)\cap N(b)$ and let $c\in R$. We observe that the function $f_1=(N_G(\{a,b\})\setminus\{c\},V(G)\setminus N_G[\{a,b\}]\cup\{c\},\{a,b\})$ is a QTRDF of weight $n-(2(k+1)-r)+5=n-2k+r+3$ on $G$. Now, the function $f_2=(V(G)\setminus(R\cup \{a,b\}),R,\{a,b\})$ is a QTRDF of weight $r+4$ on $\overline{G}$. Thus, $n+5=\gamma_{qtR}(G)+\gamma_{qtR}(\overline{G})\le n-2k+r+3+r+4=n-2k+2r+7$, which leads to $k\le r+1$. Consequently, we obtain that any vertex $b$ not in $N[a]$ has at least $k-1$ neighbors in $N(a)$.

Consider now a vertex $d\in N(a)$. Let $t$ be the number of neighbors in common between $a$ and $d$. Clearly $0\le t\le k-1$. Let $T=N(a)\cap N(d)$. We observe that the function $f_3=(N_G(\{a,d\}),V(G)\setminus N_G[\{a,d\}],\{a,d\})$ is a QTRDF of weight $n-(2k-t)+4=n-2k+t+4$ on $G$. If $n\le 2k$, then $N_G[\{a,d\}]=V(G)$ and the function $f_3$ has weight four and we can construct a QTRDF of weight six on $\overline{G}$. Thus,  $n+5=\gamma_{qtR}(G)+\gamma_{qtR}(\overline{G})\le 10$, which means $n\le 5$ but there are no graphs of order at most five satisfying the requirements at this point. Hence, we may assume $n\ge 2k+1$ and consider a vertex $d'\notin N_G(a,d)$.

Now, the function $f_4=(V(G)\setminus(T\cup \{a,d,d'\}),T\cup \{d'\},\{a,d\})$ is a QTRDF of weight $t+5$ on $\overline{G}$. Thus, $n+5=\gamma_{qtR}(G)+\gamma_{qtR}(\overline{G})\le n-2k+t+4+t+5=n-2k+2t+9$, which leads to $k\le t+2$. Consequently, we obtain that any vertex $d\in N[a]$ has at least $k-2$ neighbors in $N[a]$, which is equivalent to say that $d$ has at most two neighbors outside $N[a]$.

By using the conclusions above, we shall now count the edges between vertices in $N[a]$ and outside $N(a)$ in two directions. That is, it must happen $(k-1)(n-k-1)\le 2k$, which leads to $n\le \frac{k^2+2k-1}{k-1}$. Since $n\ge 2k+1$, we finally deduce that $k\le 3$ and, as a consequence, $k=3$ ($G$ is $3$-regular). Since $2k+1\le n\le \frac{k^2+2k-1}{k-1}$, it must be $n=7$, but there are not $3$-regular graphs of order $n=7$, and this completes the proof of item (ii).
\end{proof}


\begin{thebibliography}{99}

\bibitem{total-Roman-1} H. Abdollahzadeh Ahangar, M. A. Henning, V. Samodivkin, and I. G. Yero, Total Roman domination in graphs, \emph{Applicable Analysis and Discrete Mathematics} \textbf{10} (2016) 501--517.

\bibitem{amjadi} J. Amjadi, S. M. Sheikholeslami, and M. Soroudi, Nordhaus-Gaddum bounds for total Roman domination, \emph{Journal of Combinatorial Optimization} \textbf{35}(1) (2018) 126--133.

\bibitem{ah} M. Aouchiche and P. Hansen, A survey of Nordhaus-Gaddum type relations, {\em Discrete Applied Mathematics} {\bf 161} (2013) 466--546.

\bibitem{chambers} E. W. Chambers, B. Kinnersley, N. Prince, and D. B. West, Extremal problems for Roman domination, \emph{SIAM Journal on Discrete Mathematics} \textbf{23}(3) (2009) 1575--1586.

\bibitem{chellali2015} M. Chellali, T. Haynes, and S. Hedetniemi,  Roman and total domination, \emph{Quaestiones Mathematicae} \textbf{38} (2015) 749--757.

\bibitem{CDH04}  E. J. Cockayne, P. A. Dreyer, S M. Hedetniemi, and S. T. Hedetniemi, Roman domination in graphs. \textit{Discrete Mathematics} \textbf{278} (1-3) (2004) 11--22.

\bibitem{Dreyer} P. A. Dreyer. Applications and variations of domination in graphs. Ph. D. Thesis. New Brunswick, New Jersey, 2000.

\bibitem{hhs1}  T.W.Haynes, S.T.Hedetniemi, P.J.Slater, Domination in Graphs, Marcel Dekker, Inc., New York, NY, 1998.Zbl 0890.05002

\bibitem{hhs2} T.W.Haynes, S.T.Hedetniemi, P.J.Slater, Domination in Graphs: Advanced Topics. Marcel Dekker, Inc., New York, NY, 1998.Zbl 0883.00011

\bibitem{Henning2009} M. A. Henning, A survey of selected recent results on total domination in graphs, \emph{Discrete Mathematics} 309 (2009) 32--63.

\bibitem{HeYe_book} M. A. Henning, A. Yeo, Total Domination in Graphs. Springer Monographs in Mathematics (2013).


\bibitem{lc}  C.-H. Liu, G. J. Chang, Roman domination on strongly chordal graphs, \emph{Journal of Combinatorial Optimization} DOI 10.1007/s10878-012-9482-y

\bibitem {ng} E.A. Nordhaus and J. Gaddum, On complementary graphs, {\em American Mathematical Monthly} {\bf 63} (1956) 175--177.

\bibitem{re1} ReVelle CS (1997a) Can you protect the Roman Empire? Johns Hopkins Mag 49(2):40

\bibitem{re2} ReVelle CS (1997b) Test your solution to Can you protect the Roman Empire? Johns Hopkins Mag
                             49(3):70

\bibitem{rer} ReVelle CS, Rosing KE (2000) Defendens Imperium Romanum: a classical problem in military. Am Math
Mon 107(7):585--594

\bibitem{St99} I. Stewart, Defend the Roman Empire!, \textit{Scientific American}, December 1999, 136--138.

\bibitem{w} D. B. West, Introduction to graph theory (Second Edition) Prentice Hall USA (2001).

\end{thebibliography}
\end{document}